\documentclass[leqno,12pt]{amsart} 

\usepackage{calc}
\usepackage[left=1in,right=1in,top=1in,bottom=1in]{geometry}  %I had to print .pdf using Adobe Acrobat to keep these margins in a printed copy!!
\usepackage{graphicx}
\usepackage{amsmath,amssymb,epsfig,mathtools}
\usepackage{amsthm}
\usepackage{enumerate}
\usepackage{caption}
\usepackage{tikz}
\usepackage{bm}
\makeatletter
\renewcommand*\env@matrix[1][*\c@MaxMatrixCols c]{%
  \hskip -\arraycolsep
  \let\@ifnextchar\new@ifnextchar
  \array{#1}}
\makeatother

\newcommand{\Q}{{\mathbb Q}}

\newcommand{\Gal}{{\mathrm{Gal}}}

\newcommand{\Sel}{{\mathrm{Sel}}}
\newcommand{\dimF}{{\mathrm{dim}_{\mathbb{F}_2}}}

\newcommand{\val}{{\mathrm{val}}}

\newcommand{\p}{{ \mathfrak{p} }}

\newcommand{\Prob}{{ \mathrm{Prob} }}
\newcommand{\Res}{{ \mathrm{Res} }}

\newcommand{\QQ}{\mathbb{Q}}
\newcommand{\Cl}{\mathrm{Cl}}

\newcommand{\Disc}{\mathrm{Disc}}
\newcommand{\Aut}{\mathrm{Aut}}
\newcommand{\Avg}{\mathrm{Avg}}
\newcommand{\cO}{\mathcal{O}}
\newcommand{\SF}{\mathcal{S}_4}
\newcommand{\ST}{\mathcal{S}_3}

\newtheorem{thm}{\bf{Theorem}}
\newtheorem{theorem}{\bf{Theorem}}[section]
\newtheorem{cor}[thm]{\bf{Corollary}}
\newtheorem{corollary}[theorem]{\bf{Corollary}}

\newtheorem{proposition}[theorem]{\bf{Proposition}}

\newtheorem{lemma}[theorem]{\bf{Lemma}}

\theoremstyle{definition}
\newtheorem{definition}[theorem]{\bf{Definition}}
\newtheorem{remark}[theorem]{\bf{Remark}}

\input xy
\xyoption{all}

\usepackage[OT2,T1]{fontenc}

\DeclareSymbolFont{cyrletters}{OT2}{wncyr}{m}{n}
\DeclareMathSymbol{\Sha}{\mathalpha}{cyrletters}{"58}

\title[Two-torsion subgroups of class groups of cubic fields]{Two-torsion subgroups of class groups of cubic fields} 
\author{Zev Klagsbrun}
\address{Center for Communications Research, 4320 Westerra Court, San Diego, CA 92121}
\email{zdklags@ccrwest.org}

\begin{document}
\bibliographystyle{alpha}

\begin{abstract}
We prove a generalization of a result of Bhargava regarding the average size $\Cl(K)[2]$ as $K$ varies among cubic fields. For a fixed set of rational primes $S$, we obtain a formula for the average size of 
$\Cl(K)/\langle S \rangle[2]$ as $K$ varies among cubic fields with a fixed signature, where $\langle S \rangle$ is the subgroup of $\Cl(K)$ generated by the classes of primes of $K$ above prime in $S$.

As a consequence, we are able to calculate the average sizes of $K_{2n}(\cO_K)[2]$ for $n > 0$ and for the relaxed Selmer group $\Sel_2^S(K)$ as $K$ varies in these same families.
\end{abstract}
\maketitle

\section{Introduction}

In addition to the Davenport-Heilbronn theorem, one of the few results proven concerning the distribution of class groups of number fields is a result of Bhargava in \cite{BhargavaQuartics} and extended by Bhargava and Varma in \cite{BV} which states:
\begin{thm}
\label{thm:DHforCubics}
When ordered by absolute discriminant,
\begin{enumerate}[(i)]
\item the average size of $\Cl(K)[2]$ as $K$ ranges over totally real $\ST$-cubic fields is equal to $5/4$,
\item the average size of $\Cl(K)[2]$ as $K$ ranges over complex $\ST$-cubic fields is equal~to~$3/2$,~and 
\item the average size of $\Cl(K)^+[2]$ as $K$ ranges over totally real $\ST$-cubic fields is equal to $2$. 
\end{enumerate}
\end{thm}

Theorem \ref{thm:DHforCubics} may be thought of as an analogue of the classical Davenport-Heilbronn theorem regarding the average size of the 3-torsion subgroups of class groups of quadratic fields. In~\cite{DHQuotients}, we generalized the Davenport-Heilbronn theorem to quotients of ideal class groups of quadratic fields by the subgroup generated by the classes of primes lying above a fixed set of rational primes $S$. The goal of this work is to do the same for Theorem \ref{thm:DHforCubics}.  

Explicitly: Let $S$ be a finite set of rational primes.
For each cubic field $K$, define $\Cl(K)_S\nobreak:\nobreak=\nobreak\Cl(K)/\langle S_K\rangle$, where $S_K$ is the set of primes of $\cO_K$ lying above the primes in $S$ and $\langle S_K \rangle$ is the subgroup of $\Cl(K)$ generated by the ideal classes of the primes in $S_K$. Define $\Cl(K)^+_S$ similarly using the narrow class group $\Cl(K)^+$ of $K$.

\begin{thm}
\label{thm:DH for quotients of cubics}
When ordered by absolute discriminant,
\begin{enumerate}[(i)]
\item the average size of $\Cl(K)_S[2]$ as $K$ ranges over totally real  $\ST$-cubic fields is equal to
\linebreak
$\displaystyle{
1 + 
\frac{1}{2^{|S|+2}}
\prod_{p \in S} 
\left(
1  
+
\frac{p^2+4}{4(p^2+p+1)}
\right)
}$,
\item the average size of $\Cl(K)_S[2]$ as $K$ ranges over complex  $\ST$-cubic fields is equal to
\linebreak
$\displaystyle{
1 + 
\frac{1}{2^{|S|+1}}
\prod_{p \in S} 
\left(
1  
+
\frac{p^2+4}{4(p^2+p+1)}
\right)
}$, and
\item the average size of $\Cl(K)^+_S[2]$ as $K$ ranges over totally real  $\ST$-cubic fields is equal to
\linebreak
$\displaystyle{
1 + 
\frac{1}{2^{|S|}}
\prod_{p \in S} 
\left(
1  
+
\frac{p^2+4}{4(p^2+p+1)}
\right)
}$.
\end{enumerate}
\end{thm}

Theorem \ref{thm:DH for quotients of cubics} has the following immediate corollary.

\begin{cor}
\label{cor:posprop}
If $S$ is non-empty, then
\begin{enumerate}[(i)]
\item a positive proportion of totally real $\ST$-cubic fields $K$ have $\Cl(K)_S^+[2] = 0$,
\item a positive proportion of totally real $\ST$-cubic fields $K$ have $\Cl(K)_S^+ = \Cl(K)$, and
\item a positive proportion of totally real $\ST$-cubic fields $K$ have $S_K$-units of all possible signatures.
\end{enumerate}
In each case, the proportion of totally real $\ST$-cubic fields having the property claimed is at least $\displaystyle{
1 - 
\frac{1}{2^{|S|}}
\prod_{p \in S} 
\left(
1  
+
\frac{p^2+4}{4(p^2+p+1)}
\right)
}$.
\end{cor}

The local product
appearing in Theorem \ref{thm:DH for quotients of cubics} is a consequence of the fact that the decomposition type of any prime in $S$ will vary with $K$. By assuming a fixed decomposition type for each prime in $S$, we get a more natural answer that is consistent with Cohen and Lenstra's model for quotients of class groups --- see Remark \ref{rem:CLrem}.

\begin{thm}
\label{thm:DH for quotients of cubics local}
Let $S$ be a set of rational primes. For each $p \in S$, fix a rank 3 $\QQ_p$-algebra $R_p$ and set $r = \sum_{p \in S} (r_p - 1)$, where $r_p$ is the number of irreducible components of $R_p$. When ordered by absolute discriminant,
\begin{enumerate}[(i)]
\item the average size of $\Cl(K)_S[2]$ as $K$ ranges over totally real  $\ST$-cubic fields with \linebreak$K\nobreak\otimes\nobreak\QQ_p\nobreak\simeq\nobreak R_p$ for all $p \in S$ is equal to $1 + 2^{-(r + 2)}$,
\item the average size of $\Cl(K)_S[2]$ as $K$ ranges over complex  $\ST$-cubic fields with \linebreak$K\nobreak\otimes\nobreak\QQ_p\nobreak\simeq\nobreak R_p$ for all $p \in S$ is equal to $1 + 2^{-(r + 1)}$, and
\item the average size of $\Cl(K)^+_S[2]$ as $K$ ranges over totally real  $\ST$-cubic fields  with \linebreak$K\nobreak\otimes\nobreak\QQ_p\nobreak\simeq\nobreak R_p$ for all $p \in S$ is equal to $1 + 2^{-r}$
\end{enumerate}
\end{thm}

We are able to use Theorem  \ref{thm:DH for quotients of cubics local} to obtain results about $K_{2n}(\cO_K)$ for $n > 0$. When $S = \{2\}$, a theorem of Rognes and Weibel relates $\dimF K_{2n}(\cO_K)$ to $\Cl(K)_S$ when $n \equiv 0,1 \pmod{4}$ and to $\Cl(K)^+_S$ when $n \equiv 2,3 \pmod{4}$ --- see Theorem \ref{thm:RG thm}. As a consequence, we obtain the following result about the average size of $K_{2n}(\cO_K)[2]$ as $K$ varies over cubic fields.

\begin{thm}\label{thm:K2thm}
For $n > 0$, the average size of $K_{2n}(\cO_K)[2]$  as $K$ ranges over totally real (resp.~complex) $\ST$-cubic fields is as follows: 
\begin{center}
\begin{tabular}{c||c|c}
& $\textup{$K$ totally real}$ & $\textup{$K$ complex}$ \\ \hline \hline
$\textup{$n \equiv 0 \pmod{4}$}$ & $59/28$ & $33/14$ \\ \hline
$\textup{$n \equiv 1 \pmod{4}$}$ & $118/7$ & $33/7$ \\ \hline
$\textup{$n \equiv 2 \pmod{4}$, $n \equiv 3 \pmod{4}$}$ & $20/7$ & $33/14$ \\ 
\end{tabular}
\end{center}
\end{thm}

\begin{remark}
Theorem \ref{thm:K2thm} is an analogue of Theorem 1.2 in the author's joint work with Jordan, Poonen, Skinner, and Zaytman which proves a similar result about the average size of $K_{2n}(\cO_K)[3]$ as $K$ varies over quadratic fields~\cite{JKPSZ}. 
\end{remark}

\begin{remark}
We restrict our attention to even indexed $K$-groups, since for $n$ odd, $K_n(\cO_K)[2]$ is entirely determined by the residue class of $n \pmod{8}$ and the number of real places of $K$ (see Theorem 0.7 in~\cite{RG}). 
\end{remark}

We also obtain distribution results about $\Sel_2^S(K)$, the relaxed $2$-Selmer group of $K$ (defined in Section~\ref{sec:Selmer Groups}).
\pagebreak

\begin{thm}
\label{thm:DH quotients for Sel_2^S}
When ordered by absolute discriminant,
\begin{enumerate}[(i)]
\item the average size of $\Sel_2^S(K)$ as $K$ ranges over totally real $\ST$-cubic fields is equal to 
\linebreak
$\displaystyle{2^{|S|+1} + 2^{|S|+3} \prod_{p \in S} \left (2 - \frac{1}{p^2 + p + 1} \right )}$
and
\item the average size of $\Sel_2^S(K)$ as $K$ ranges over complex $\ST$-cubic fields is equal to
\linebreak
$\displaystyle{2^{|S|+1} + 2^{|S|+2} \prod_{p \in S} \left (2 - \frac{1}{p^2 + p + 1} \right )}$. 
\end{enumerate}
\end{thm}

\begin{remark}
\label{rem:CLrem}
In Theorem \ref{thm:DH for quotients of cubics local}, $r_p$ is a the number of primes above $p$ in $K$. If the primes above $p$ take classes uniformly at random in $\Cl(K)$ subject only to the relation arising from the factorization of $p\cO_K$ and all of the primes in $S$ behave independently, then the subgroup $\langle S_K \rangle \le \Cl(K)$ may be throught of as a group generated by $r$ elements chosen uniformly at random from $\Cl(K)$. It is therefore natural to expect that for any finite abelian 2-group $H$, the probability $\Prob \left (\Cl(K)/\langle S \rangle[2^\infty] \simeq H \right )$ is equal to what Cohen and Lenstra dub the $u$-probability of $H$ with $u = r + r_\infty$, where $r_\infty$ is equal to $2$, $1$, and $0$ in cases (i), (ii), and (iii) of the theorem respectively. The average sizes appearing Theorem \ref{thm:DH for quotients of cubics local}  are precisely the $u$-averages for these values of $u$ \cite{CL}.
\end{remark}

\begin{remark}
\label{rem:nocyc}
All of the results we present are stated in terms of $\ST$-cubic fields. However, the results remain correct and can be proven with modified versions of the current proofs even if we remove the $\ST$ assumption.

We have nonetheless chosen to maintain the $\ST$-assumption throughout the paper for the purposes of clarity. Instead, we present the following argument showing that the total contribution to the average size of $\Cl(K)[2]$ coming from cyclic cubic fields as $K$ varies among all totally real cubic fields is zero.

It is well-known that the total number of cyclic cubic fields of discriminant less than $X$ grows like $O(X^{1/2})$ \cite{Cohn}. By a result of Wong \cite{Wong}, the size of $\Cl(K)[2]$ in any cyclic cubic field $K$ is bounded by $O(|\Disc(K)|^{3/8 + \epsilon})$ for any $\epsilon > 0$ (see also \cite{BSTTTZ} for a better bound).
As a result, the combined number of elements in $\Cl(K)[2]$ from all cyclic cubic fields with discriminant less than $X$ is bounded by $O(X^{7/8 + \epsilon})$. Since the number of totally real cubic fields of discriminant less than $X$ grows like $O(X)$, the total contribution to the average size of $\Cl(K)[2]$ coming from cyclic cubic fields must be zero.
\end{remark}

\subsection{Methods and Organization}
The core idea behind Theorem \ref{thm:DHforCubics} is how to use the geometry of numbers to count $\SF$-quartic fields. The application to class groups as it originally appears in \cite{BhargavaQuartics} arises from a bijection established by Heilbronn and slightly refined by Bhargava between the set of non-trivial two-torsion elements in the class group $\Cl(K)$ of an $\ST$-cubic field and the set of
\textit{nowhere overramified} isomorphism classes of $\SF$-quartic fields $L$ with \textit{resolvent field} $K$ (see Definitions \ref{def:res}~and~\ref{def:over}) that have a real place.

We establish a similar bijection in Section~\ref{sec:Class Field Theory} between the set of non-trivial two-torsion elements in $\Cl(K)_S$ and the set of isomorphism classes of $\SF$-quartic fields $L$ with \textit{resolvent field} $K$ satisfying certain local conditions. Section~\ref{sec:Counting Fields} presents recent results of Bhargava used for counting the number of such fields having bounded discriminant~\cite{BhargavaGeometricSieve} and Section~\ref{sec:Local Conditions} then describes how to apply to these techniques to prove Theorem \ref{thm:DH for quotients of cubics}. Section \ref{sec:Selmer Groups} defines the relaxed $2$-Selmer group of $K$ and develops the machinery needed to prove Theorem \ref{thm:DH quotients for Sel_2^S}. Finally, we prove Theorem \ref{thm:K2thm} in Section \ref{sec:Kgroups}.

\subsection{Acknowledgement}

I would like to thank Frank Thorne for suggesting the use of the Artin relation in the proof of Lemma \ref{lem:artrel} and for pointing out the result of Wong mentioned in Remark \ref{rem:nocyc}. I would also like to thank Alyson Deines for a number of useful suggestions.

\subsection{Notation}
We will use the following notation throughout this paper:
\begin{itemize}
\item $S$ will be a set of rational primes.
\item $K$ will be a cubic field.
\item $\cO_K$ will be the ring of integers of $K$.
\item $S_K$ will denote the set of primes of $\cO_K$ lying above primes in $S$.
\item $\cO_{K,S}^\times$ will denote the $S_K$-units of $K$.
\item $\Sel_2^S(K)$ will be the 2-Selmer group of $K$ relaxed at the primes in $S_K$.
\item $\Cl(K)$ will be the ideal class group of $\cO_K$.
\item $\Cl(K)_S$ will be the quotient $\Cl(K)/\langle S_K\rangle$, where $\langle S_K \rangle $ is the subgroup of $\Cl(K)$ generated by primes in $S_K$.
\end{itemize}

\section{Class Field Theory}
\label{sec:Class Field Theory}
Bhargava's results about $\Cl(K)[2]$ in~\cite{BhargavaQuartics} rely on a correspondence of Heilbronn detailed in~\cite{H}. We briefly describe this correspondence before establishing a similar correspondence for $\Cl(K)_S[2]$ in Proposition \ref{prop:DavCorQuotient}.

\begin{definition}
\label{def:res}
Given an $\SF$-quartic field $L$, the \textit{cubic resolvent field} $\Res(L)$ is the unqiue (up to isomorphism) cubic subfield of the Galois closure $N$ of $L/\QQ$. 
\end{definition}

While the resolvent field of $L$ is unique, non-isomorphic $L$ may share the same resolvent field.

Following Section 3.1 in \cite{BhargavaQuartics}, we make the following definition.

\begin{definition}
\label{def:over}
Let $p$ be a rational prime. A quartic field $L$ is called \textit{overramified at $p$} if $L \otimes \QQ_p$ is either irreducible and ramified or the direct sum of two ramified fields. The field $L$ is called \textit{nowhere overramified} if $L$ is not overramified at $p$ for any prime $p$.
\end{definition}

\begin{remark}
Note that unlike Bhargava's definition in \cite{BhargavaQuartics}, our definition of nowhere overramified does not include any restriction on the ramification of $L$ at infinity.
\end{remark}

The fields $L$ and $\Res(L)$ have the same discriminant precisely when $L$ is nowhere overramified.

\begin{proposition}
\label{prop:DavCor}
Let $K$ be an $\ST$-cubic field.
\begin{enumerate}[(A)]
\item The following are in bijective correspondence.
\begin{enumerate}[(i)]
\item \label{it:DF1}The set of index two subgroups of $\Cl(K)$.
\item \label{it:DF2}The set of unramified quadratic extensions $F$ of $K$.
\item \label{it:DF3}The set of isomorphism classes of nowhere overramified $\SF$-quartic fields $L$ having a real place with $\Res(L) = K$. 
\end{enumerate}
\item The following are in bijective correspondence.
\begin{enumerate}[(i)]
\item The set of index two subgroups of $\Cl(K)^+$.
\item The set of quadratic extensions $F$ of $K$ that are unramified at all finite places.
\item The set of isomorphism classes of nowhere overramified $\SF$-quartic fields $L$ with $\Res(L) = K$.
\end{enumerate}
\end{enumerate}
\end{proposition}

\begin{proof}
For both (A) and (B), the correspondence (i) $\leftrightarrow$ (ii) is class field theory. The correspondence between (ii) and (iii) in (B) is detailed in \cite{H} and the correspondence between (ii) and (iii) in (A) comes from restricting this correspondence to the sets in (A). 
While we do not include a proof of the equivalence between (ii) and (iii) here, we will however describe the maps yielding the correspondence.

We begin with the map (iii) $\rightarrow$ (ii). Let $N$ be the Galois closure of $L/\QQ$. By assumption, $\Gal(N/\QQ) \simeq \SF$. The group $\SF$ contains three distinct $D_4$ subgroups, all of which are conjugate. The resolvent field $K = \Res(L)$ may be taken to be the fixed field of any of these $D_4$ subgroups. The group $D_4$ contains two distinct Klein four subgroups, only one of which contains a transposition when $D_4$ is viewed as a subgroup of $\Gal(N/\QQ) \simeq \SF$. Letting $V$ be that subgroup, the extension $F/K$ is given by the fixed field of $V$.

We next describe the map (ii) $\rightarrow$ (iii). While the field $F/\QQ$ is not Galois, the lack of ramification (at finite primes) in $F/K$ forces its Galois closure $N$ to be an $\SF$-extension of $\QQ$. The group $\SF$ contains four distinct $\ST$ subgroups, all of which are conjugate. The field $L$ may be taken to be the fixed field of any of these $\ST$ subgroups.
\end{proof}

\begin{remark}
In both Proposition \ref{prop:DavCor} and Proposition \ref{prop:DavCorQuotient} which follows, if $K$ is taken to be a complex cubic field, then $\Cl(K) = \Cl(K)^+$ and parts (A) and (B) are equivalent.
\end{remark}

\begin{proposition}
\label{prop:DavCorQuotient}
Let $K$ be an $\ST$-cubic field.
\begin{enumerate}[(A)]
\item The following are in bijective correspondence.
\begin{enumerate}[(i)]
\item The set of index two subgroups of $\Cl(K)_S$.
\item The set of unramified quadratic extensions $F$ of $K$ in which all primes in $S_K$ split completely.
\item The set of isomorphism classes of nowhere overramified $\SF$-quartic fields $L$ such that $\Res(L) = K$, $L$ has a real place, and $L \otimes \QQ_p$ has a component equal to $\QQ_p$ for all $p \in S$.
\end{enumerate}
\item The following are in bijective correspondence.
\begin{enumerate}[(i)]
\item The set of index two subgroups of $\Cl(K)^+_S$.
\item The set of quadratic extensions $F$ of $K$ that are unramified at all finite places and in which all primes in $S_K$ split completely.
\item The set of isomorphism classes of nowhere overramified $\SF$-quartic fields $L$ such that $\Res(L) = K$ and $L \otimes \QQ_p$ has a component equal to $\QQ_p$ for all $p \in S$.
\end{enumerate}
\end{enumerate}
\end{proposition}

\begin{proof}
For both (A) and (B), the correspondence (i) $\leftrightarrow$ (ii) is class field theory. The equivalences (ii) $\leftrightarrow$ (iii) in (A) and (B) follow from the similar equivalences in Proposition \ref{prop:DavCor} and Lemma \ref{lem:artrel} appearing below.
\end{proof}

\begin{lemma}
\label{lem:artrel}
Let $K$ be an $\ST$-cubic field field and $F/K$ a quadratic extension unramified at all finite places. If 
$L$ is the $\SF$-quartic field corresponding to $F/K$ under Heilbronn's correspondence as in Proposition \ref{prop:DavCor}, then $L \otimes \QQ_p$  has a component equal to $\QQ_p$ if and only if all primes of $K$ above $p$ split in $F/K$.
\end{lemma}
\begin{proof}
Let $N$ be the Galois closure of $L/\QQ$. If $L \otimes \QQ_p$  has a component equal to $\QQ_p$, then $\Gal(N_P/\Q_p) \le \ST$ for all primes $P \mid p$ of $N$. Letting $\p$ be a prime of $K$ above $p$, we have $\Gal(N_P/K_\p) \le \Gal(N/K)$. Since $\Gal(N_P/K_\p)$ must simultaneously embed into a copy of $\ST$ and a copy of $D_4$ inside of $\SF$, we find that $\Gal(N_P/K_\p)$ must be a subgroup of the Klein four group $\Gal(N/F)$. As a result, $\p$ splits in $F/K$.

For the opposite direction, we rely on the Artin relation of zeta functions (see \cite{H}, for example)
\begin{equation}
\label{eq:artrel}
\zeta_L(s) = \frac{\zeta(s)\zeta_F(s)}{\zeta_K(s)}.
\end{equation}
If all primes of $K$ above $p$ split in $F/K$, then we have $\zeta^{(p)}_F(s) = \zeta^{(p)}_K(s)$, where $\zeta^{(p)}_{*}(s)$ denotes the part of the Euler product for $\zeta_{*}(s)$ coming from primes above $p$. As a result, if all primes of $K$ above $p$ split in $F/K$, then \eqref{eq:artrel} yields 
\begin{equation}
\label{eq:artrel2}
\zeta^{(p)}_L(s) = (1 - p^{-s})^{-1}\zeta^{(p)}_K(s).
\end{equation}
Observe that if $L \otimes \QQ_p$ contains a ramified component, then $K \otimes \QQ_p$ must contain a ramified component of the same degree. As a result, \eqref{eq:artrel2} holds even if we restrict to the Euler factors coming from unramified primes. As a result, the Euler product for $\zeta_L(s)$ contains a factor of $(1 - p^{-s})^{-1}$ coming from an unramified prime above $p$, and the $L \otimes \QQ_p$ has a component equal to $\QQ_p$.
\end{proof}

We therefore get the following corollary:
\begin{corollary}
\label{cor:DavCor for quotients}
If $K$ is an $\ST$-cubic field, then
\begin{enumerate}[(i)]
\item 
$
\left |
\Cl(K)_S[2]
\right |
=
1 +
\left |
\left \{
\begin{matrix}
\text{nowhere overramified }
\SF\text{-quartic fields }
L
\text{ (up to iso.)}\\
\text{such that }
\Res(L) = K
\text{, }
L \text{ has a real place, and}\\
L \otimes \QQ_p
\text{ has a component equal to }
\QQ_p
\text{ for all }
p \in S
\end{matrix}
\right \}
\right |
$
and

\vspace{0.1in}

\item$
\left |
\Cl(K)^+_S[2]
\right |
=
1 +
\left |
\left \{
\begin{matrix}
\text{nowhere overramified }
\SF\text{-quartic fields }
L \\
\text{(up to iso.) }
\text{such that }
\Res(L) = K
\text{ and }
L \otimes \QQ_p \\
\text{ has a component equal to }
\QQ_p
\text{ for all }
p \in S
\end{matrix}\right \}
\right |
$.
\end{enumerate}
\end{corollary}
\begin{proof}
Since $\Cl(K)_S$ is a finite abelian group, the number of index two subgroups of $\Cl(K)_S$ is the same as the number of non-trivial two-torsion elements of $\Cl(K)_S$. By Proposition \ref{prop:DavCorQuotient}, this is equal to the number of quartic fields $L$ such that $\Res(L)=K$, $L$ has a real place and $L \otimes \QQ_p$ has $\QQ_p$ component for all $p \in S$. The result for $\Cl(K)^+_S$ follows similarly.
\end{proof}

\section{Counting Fields}
\label{sec:Counting Fields}

In order to prove Theorems \ref{thm:DH for quotients of cubics} and \ref{thm:DH for quotients of cubics local}, we will need to be able to count the number of quartic fields of bounded discriminant satisfying a given set of local conditions.

For a set $\Sigma_p$ of $\QQ_p$-algebras, define $\mu_p(\Sigma_p)$ as
$$\mu_p(\Sigma_p) := \sum_{R \in \Sigma_p} \frac{p-1}{p} \cdot \frac{1}{\Disc_p(R)} \cdot \frac{1}{|\Aut(R)|},$$
where $\Disc_p(R)$ is the $p$-part of  the discriminant of $R$.

For each $p \in S$, let $\Sigma_p$ be a set of non-overramified rank 4 $\QQ_p$ and set $\Sigma = (\Sigma_p)_{p\in S}$. For each $i \in \{0,2,4\}$, define $N^{(i)}(X, \Sigma)$ to be the number of nowhere overramified $\SF$-quartic fields $L$ (up to isomorphism) such that $|\Disc(L)| < X$, $L$ has $i$ real places, and $L \otimes \QQ_p \in \Sigma_p$ for all $p \in S$.

We then have the following specialization of a theorem of Bhargava.

\begin{theorem}[Theorem 1.3 in \cite{BhargavaGeometricSieve}]
\label{thm:quarticfldcount}
For each $i \in \{0,2,4\}$, 
\begin{equation}
N^{(i)}(X, \Sigma)
=
\frac{1}{2n_i\zeta(3)} \prod_{p \in S} \frac{\mu_p(\Sigma_p)}{\bm{\mu}_p} \cdot X + o(X)
\end{equation}
where
$\bm{\mu}_p = 1 - \frac{1}{p^3}$ and $n_i = \left \{ \begin{matrix}8 & \text{if } i = 0\\ 4 &\text{if } i = 2 \\24 & \text{if } i = 4\end{matrix} \right . .$
\end{theorem}
\begin{proof}
This will follow from Theorem 1.3 in \cite{BhargavaGeometricSieve}. For each prime $p$, define $\hat \Sigma_p$ to be the set of all rank 4 non-overramified $\QQ_p$-algebras. An easy computation shows that $\mu_p(\Sigma_p) = 1 - \frac{1}{p^3}$. We then set $\hat \Sigma = (\hat \Sigma_p)_{p}$.

For each prime $p \in S$, let $\Sigma_p$ be as above and for $p \not \in S$, set $\Sigma_p = \hat \Sigma_p$. Applying Theorem~1.3 in \cite{BhargavaGeometricSieve}, we then get
$$
N^{(i)}(X, \Sigma) =
\frac{1}{n_i}\prod_p \mu_p(\hat \Sigma_p)
\prod_{p \in S} \frac{\mu_p(\Sigma_p)}{\mu_p(\hat \Sigma_p)}
=
N^{(i)}(X, \hat \Sigma) \prod_{p \in S} \frac{\mu_p(\Sigma_p)}{\mu_p(\hat \Sigma_p)}
$$

However, $N^{(i)}(X, \hat \Sigma)$ is simply the number of nowhere totally ramified $\SF$-quartic fields $L$ (up to isomorphism) such that $|\Disc(L)| < X$ and $L$ has $i$ real places. By Lemma 27 in \cite{BhargavaQuartics}, this is known to be $\frac{1}{2n_i\zeta(3)} \cdot X + o(X)$, so the result follows.
\end{proof}

We may similarly count the number of cubic fields of bounded discriminant satisfying a set of local conditions.
For a set of local conditions $\Sigma = (\Sigma_p)_{p\in S}$ where each $\Sigma_p$ is a set of rank 3 $\QQ_p$-algebras and each $i \in \{1,3\}$, define $M^{(i)}(X, \Sigma)$ to be the number of $\ST$-cubic fields $K$ (up to isomorphism) such that $|\Disc(K)| < X$, $K$ has $i$ real places, and $K \otimes \QQ_p \in \Sigma_p$ for all $p \in S$.

\begin{theorem}[Theorem 1.3 in \cite{BhargavaGeometricSieve}]
\label{thm:cubicfldcount}
For each $i \in \{1,3\}$, 
\begin{equation}
M^{(i)}(X, \Sigma)
=
\frac{1}{2m_i \zeta(3)} \prod_{p \in S} \frac{\mu_p(\Sigma_p)}{\bm{\mu}_p} \cdot X + o(X)
\end{equation}
where $m_1 = 2$, $m_3 = 6$, and $\bm{\mu}_p$ is as in Theorem \ref{thm:quarticfldcount}.
\end{theorem}
The proof of Theorem \ref{thm:cubicfldcount} is extremely similar to that of Theorem \ref{thm:quarticfldcount} so we omit it.

\section{Dealing With Local Conditions}
\label{sec:Local Conditions}

In general, if $K$ is an $\ST$-cubic field and $L$ an $\SF$-quartic field such that the resolvent $\Res(L)=K$, then $K \otimes \QQ_p$ does not determine $L \otimes \QQ_p$. However, if we further assume that $L \otimes \QQ_p$ has a $\QQ_p$ component, then this is no longer the case.

\begin{lemma}
\label{lem:addQp}
Let $K$ be an $\ST$-cubic field and $L$ an $\SF$-quartic field such that the resolvent $\Res(L)=K$. If $p$ is a prime such that $L \otimes \QQ_p$ has a component equal to $\QQ_p$, then \linebreak $L \otimes \QQ_p \simeq \left( K \otimes \QQ_p \right)  \oplus \QQ_p$.
\end{lemma}
\begin{proof}
Let $N$ be the Galois closure of $L/\QQ$. Since $L \otimes \QQ_p$ has a component equal to $\QQ_p$, we have $\Gal(N_P/\QQ_p) \le \ST$ for any prime $P$ of $N$ above $P$. 

Let $\widetilde{K}$ be the Galois closure of $K/\QQ$ contained in $N/\QQ$ and let $\p$ be any prime above $p$ in $\widetilde{K}$. The Galois group of $N/\widetilde{K}$ is the unique order four normal subgroup $V$ of $\SF$. Therefore, if $P$ is any prime of $N$ above $\p$, we have $\Gal(N_P/\widetilde{K}_\p) \le V$. However, the subgroups $V$ and $\ST$ of $\SF$ intersect trivially. As a result, we have $\Gal(N_P/\widetilde{K}_\p) = 0$ and $\p$ splits completely in $N/\tilde K$.

We therefore see that $N \otimes \QQ_p = (\widetilde{K} \otimes \QQ_p )^4$. Taking $\Gal(N/L)$ invariants, we get that $L \otimes \QQ_p =  \left( K \otimes \QQ_p \right) \oplus \QQ_p$.
\end{proof}

Lemma \ref{lem:addQp} motivates the following definition. Given a rank 3 $\QQ_p$-algebra $R_p$, we define a rank 4 $\QQ_p$-algebra $\widetilde{R_p}$ as $\widetilde{R_p} = R_p \oplus \QQ_p$. We then get the following corollary.

\begin{proposition}
\label{prop:avgclassgroupsizeloc}
Let $\Sigma = (\Sigma_p)_{p\in S}$ where each $\Sigma_p$ is a set of rank 3 $\QQ_p$-algebras. Set  
$\widetilde{\mu}(\Sigma) 
=
\prod_{p \in S} {\mu_p(\widetilde{\Sigma_p})}/{\mu_p(\Sigma_p)},
$
where $\widetilde{\Sigma_p} = \{ \widetilde{R_p} : R_p \in \Sigma_p\}$. Then
\begin{enumerate}[(i)]
\item the average size of $\Cl(K)_S[2]$ as $K$ ranges over totally real $\ST$-cubic fields with \linebreak$K \otimes \QQ_p \in \Sigma_p$ for all $p \in S$ is equal to $1 + \frac{1}{4}\widetilde{\mu}(\Sigma)$
\item the average size of $\Cl(K)_S[2]$ as $K$ ranges over complex $\ST$-cubic fields with \linebreak $K \otimes \QQ_p \in \Sigma_p$ for all $p \in S$ is equal to, and $1 + \frac{1}{2}\widetilde{\mu}(\Sigma)$
\item the average size of $\Cl(K)^+_S[2]$ as $K$ ranges over totally real $\ST$-cubic fields  with \linebreak $K \otimes \QQ_p \in \Sigma_p$ for all $p \in S$ is equal to $1 + \widetilde{\mu}(\Sigma)$.
\end{enumerate}
\end{proposition}

\begin{proof}
By combining Corollary \ref{cor:DavCor for quotients} with Lemma \ref{lem:addQp}, we have
$$\sum_{\substack{K \text{ cubic, } \\ K\otimes \QQ_p \simeq R_p \text{ for all } p\in S \\  0 < \Disc(K) < X}} |\Cl(K)_S[2]| = M^{(3)}(X, \Sigma) + N^{(4)}(X,\widetilde{\Sigma}).$$ Part (i) then follows from Theorems \ref{thm:quarticfldcount} and \ref{thm:cubicfldcount}. Part (ii) follows from an essentially identical calculation.

For part (iii), again by Corollary \ref{cor:DavCor for quotients} combined with Lemma \ref{lem:addQp}, we have
$$\sum_{\substack{K \text{ cubic, } \\ K\otimes \QQ_p \simeq R_p \text{ for all } p\in S \\  0 < \Disc(K) < X}} |\Cl(K)^+_S[2]| = M^{(3)}(X, \Sigma) + N^{(4)}(X,\widetilde{\Sigma}) +  N^{(0)}(X,\widetilde{\Sigma}).$$ The result then follows from Theorems \ref{thm:quarticfldcount} and \ref{thm:cubicfldcount}.
\end{proof}

The quotients $\mu_p(\widetilde{\Sigma_p})/{\mu_p(\Sigma_p)}$ have a nice formula when either $\Sigma_p = \{ R_p \}$ or $\Sigma_p$ contains all rank 3 $\QQ_p$-algebras (up to isomorphism).

\begin{lemma}
\label{lem:SigpRp}
If $\Sigma_p = \{ R_p \}$, then $\mu_p(\widetilde{\Sigma_p})/{\mu_p(\Sigma_p)} = 2^{-(r_p - 1)}$, where $R_p$ is the number of irreducible components of $R_p$.
\end{lemma}
\begin{proof}
Since $\widetilde{R_p} = R_p \oplus \QQ_p$, we have $\Disc_p(\widetilde{R_p}) = \Disc_p(R_p)$, so $$\mu_p(\{\widetilde{R_p}\})/{\mu_p(\{ R_p \})}= |\Aut(R_p)|/|\Aut(\widetilde R_p)| = n!/(n+1)! = 1/(n+1),$$ where $n$ is the number of components of $R_p$ equal to $\QQ_p$. Examination shows that this is equal to $2^{-(r_p-1)}$.
\end{proof}

\pagebreak
\begin{lemma}
\label{lem:SigpAll}
If $\Sigma_p$ contains all rank 3 $\QQ_p$-algebras, then $\mu_p(\widetilde{\Sigma_p})/{\mu_p(\Sigma_p)} =\frac{1}{2}\left( 1 + \frac{p^2+4}{4(p^2 + p + 1)}\right)$.
\end{lemma}
\begin{proof}
By Lemma \ref{lem:SigpRp}, we have $\mu_p (\{ \widetilde{R_p} \})/\mu_p(\{ R_p \}) = 2^{-(r-1)}$ for each $R_p$ with $r$ irreducible components. Letting $\Sigma_{p,r}$ denote the set of all rank 3 $\QQ_p$-algebras with $r$ irreducible components, we get
$$\mu_p(\widetilde{\Sigma_p}) = \sum_{R_p \in \Sigma_p} \mu_p(\widetilde{\Sigma_{p,r}})= \sum_{r = 1}^3  2^{-(r-1)}\mu_p(\Sigma_{p,r}).$$ 
Calculating $\mu_p(\Sigma_{p,r})$ for each $r \in \{1,2,3\}$, we find $$\mu_p(\Sigma_{p,1}) = \frac{p^3 - p^2 + 3p -3}{3p^3},  \mu_p(\Sigma_{p,2}) = \frac{p^2+p-2}{2p^2}, \text{ and } \mu_p(\Sigma_{p,3}) =\frac{p-1}{6p}.$$ As a result, we get $\mu_p(\widetilde{\Sigma_p}) = \frac{p^3 - p^2 + 4p - 8}{8p^3}$, so $$\mu_p(\widetilde{\Sigma_p})/{\mu_p(\Sigma_p)} = \frac{p^3 - p^2 + 4p - 8}{8p^3} \cdot \frac{p^3}{p^3-1} =  \frac{5p^2 + 4p + 8}{8(p^2+p+1)} = \frac{1}{2}\left( 1 + \frac{p^2+4}{4(p^2 + p + 1)}\right).$$
\end{proof}

We are now able to prove Theorems \ref{thm:DH for quotients of cubics} and \ref{thm:DH for quotients of cubics local}.

\begin{proof}[Proof of Theorem \ref{thm:DH for quotients of cubics}]
For each $p \in S$, let $\Sigma_p$ be the set of all rank 3 $\QQ_p$ algebras. The result then follows from combining Proposition \ref{prop:avgclassgroupsizeloc} with Lemma \ref{lem:SigpAll}.
\end{proof}

Corollary \ref{cor:posprop} will now follow from Theorem \ref{thm:DH for quotients of cubics}.

\begin{proof}[Proof of Corollary \ref{cor:posprop}]
Let $\lambda$ be the proportion of totally real $\ST$-cubic fields $K$ having \linebreak $\Cl(K)_S^+[2]~=~0$. We then have $\Avg(\Cl(K)_S^+[2] ) \ge \lambda + 2\cdot (1-\lambda) = 2 - \lambda$. By Theorem \ref{thm:DH for quotients of cubics}, we have $$\Avg(\Cl(K)_S^+[2]) = \displaystyle{
1 + 
\frac{1}{2^{|S|}}
\prod_{p \in S} 
\left(
1  
+
\frac{p^2+4}{4(p^2+p+1)}
\right)
},$$ so it follows that 
\begin{equation}
\label{eq:lambdaineq}
\lambda \ge 1 - 
\frac{1}{2^{|S|}}
\prod_{p \in S} 
\left(
1  
+
\frac{p^2+4}{4(p^2+p+1)}
\right).
\end{equation}
If $S$ is non-empty, then the right-hand side of \eqref{eq:lambdaineq} is at least $5/14$, proving (i).

To see (ii), observe that the kernel of the surjection $\Cl(K)_S^+ \rightarrow \Cl(K)_S$ is a subgroup of $\Cl(K)_S^+$ having order at most two. As a result, if $\Cl(K)_S^+$ has odd order, then\linebreak $\Cl(K)_S^+ \nobreak=\nobreak \Cl(K)_S$. The result then follows from (i).

Finally, we note that $K$ has $S_K$-units of all signatures if and only if $\Cl(K)_S^+ = \Cl(K)_S$. Part (iii) of the corollary therefore follows from (ii).
\end{proof}

\begin{proof}[Proof of Theorem \ref{thm:DH for quotients of cubics local}]
For each $p \in S$, let $\Sigma_p = \{ R_p\}$. By Lemma \ref{lem:SigpRp}, for each $p \in S$, we have $\mu_p(\widetilde{\Sigma_p})/{\mu_p(\Sigma_p)}\nobreak=\nobreak2^{-(r_p - 1)}$. Letting $r = \sum_{p \in S} (r_p - 1)$, we have $\widetilde{\mu}(\Sigma) = 2^{-r}$ and the result follows from Proposition \ref{prop:avgclassgroupsizeloc}.
\end{proof}

\section{Selmer Groups}
\label{sec:Selmer Groups}

\begin{definition}
\label{def:Relaxed Selmer}
Let $S$ be a set of rational primes. The $2$-Selmer group of $K$ relaxed at $S$, denoted $\Sel_2^S(K)$ is defined as 
\[
\Sel_2^S(K) := 
\left \{
\alpha \in K^\times/(K^\times)^2
:
\val_\p(\alpha) \equiv 0 \pmod{2}
\text{ for all }
\p \not \in S_K
\right \}.
\]
where $S_K$ is the set of primes of $K$ lying about $S$.
\end{definition}

A standard result (see Section 8.3.2 of \cite{Cohen-Number Theory vol 1}, for example) shows that $\Sel_2^S(K)$ sits in the short exact sequence
\begin{equation}
\label{eq:Selmer sequence}
0 \rightarrow 
\cO_{K,S}^\times/(\cO_{K,S}^\times)^2
\rightarrow 
\Sel_2^S(K)
\rightarrow
\Cl(K)_S[2]
\rightarrow 0,
\end{equation}
where $\cO_{K,S}^\times$ is the $S_K$ units of $\cO_K$.

Define a function $\nu_{S}$ on $\ST$-cubic fields by $\nu_S(K) = \sum_{p \in S} r_p(K)$, where $r_p(K)$ is the number of irreducible components of $K \otimes \QQ_p$. Dirichlet's unit theorem then tells us that $|\Sel_2^S(K)|\nobreak=\nobreak2^{\nu_S(K) + 3} |\Cl(\cO_K)_S[2]|$ if $K$ is totally real  and $|\Sel_2^S(K)| = 2^{\nu_S(K) + 2}\nobreak|\Cl(\cO_K)_S[2]|$ if $K$ is complex.

\subsection{An averaging result for $\nu_S(K)$}

To compute the average size of $\Sel_2^S(K)$, we would there like to calculate the average value of $2^{\nu_S(K)}|\Cl(K)_S[2]|$. We begin with the following lemma.

\begin{lemma}
\label{lem:avglem}
When ordered by absolute discriminant, the average value of $2^{\nu_S(K)}$ as  $K$ ranges over totally real (resp. complex) $\ST$-cubic fields is $\displaystyle{2^{|S|}\prod_{p \in S} \left (2 - \frac{1}{p^2 + p + 1} \right )}$.
\end{lemma}

\begin{proof}
We will only prove the totally real case, since the complex case is identical.  We proceed by induction on the set $S$. For the base case, consider the case where $S$ contains a single prime $p$.

For each $r \in \{1,2,3\}$, define $\rho_p(r)$ to be the proportion of totally real $\ST$-cubics $K$ such that $r_p = r$. By Theorem \ref{thm:cubicfldcount}, we have $\rho_p(r) = \frac{\mu_p(\Sigma_{p,r})}{\bm{\mu}_p}$ where $\Sigma_{p,r}$ is as in the proof of Lemma \ref{lem:SigpAll} and $\bm{\mu}_p = 1 - \frac{1}{p^3}$. We then get
\begin{multline}
\label{eq:avgeq3}
\Avg\left (2^{r_p}\right ) 
=  \sum_{r = 1}^{3} 2^{r} \cdot \rho_p(r) 
= \frac{p^3}{p^3-1} \left ( 2 \cdot \frac{p^3 - p^2 + 3p -3}{3p^3} +  4 \cdot \frac{p^2+p-2}{2p^2} +  8 \cdot \frac{p-1}{6p}\right ) 
\\
= 4 - \frac{2}{p^2 + p + 1}= 2 \cdot  \left (2 - \frac{1}{p^2 + p + 1} \right ) .
\end{multline}

For the inductive step, let $S^\prime = S \cup  \{p^\prime\}$ for some prime $p^\prime \not \in S$. We then have 
\begin{equation}
\Avg\left(2^{\nu_{S^\prime}(K)}\right) 
= \sum_{s} \left( 2^s \cdot \Prob(\nu_S(K) = s) \cdot  \sum_{n = 1}^3 2^r \cdot \Prob\left (r_{p^\prime} = r | \nu_S(K) = s \right) \right ).
\end{equation}

By Theorem \ref{thm:cubicfldcount}, the splitting type of $K \otimes \QQ_{p^\prime}$ is independent from the splitting type of $K \otimes \QQ_{p}$ for all $p \in S$, and therefore independent of $\nu_S(K)$. As a result, we get
\begin{multline}
\label{eq:indstep}
\Avg\left(2^{\nu_{S^\prime}(K)}\right) 
=\left( \sum_{s}2^{s} \cdot \Prob(\nu_S(K) = s)\right )\left(\sum_{n = 1}^3 2^r \cdot \rho_{p^\prime}(r) \right )
\\
=\Avg\left(2^{\nu_{S}(K)}\right)  \left(\sum_{n = 1}^3 2^r \cdot \rho_{p^\prime}(r) \right) 
= \Avg\left(2^{\nu_{S}(K)}\right)  \cdot 2 \cdot  \left (2 - \frac{1}{{p^\prime}^2 + p^\prime + 1} \right ),
\end{multline}
where the final equality follows from \eqref{eq:avgeq3}. By the inductive hypothesis, we have
\begin{equation}
\label{eq:indhyp}
\Avg\left(2^{\nu_{S}(K)}\right) = 2^{|S|} \prod_{p \in S} \left (2 - \frac{1}{p^2 + p + 1} \right ) 
\end{equation}
Combining \eqref{eq:indstep} with \eqref{eq:indhyp} yields the result.
\end{proof}

\subsection{Averaging for fixed $\nu_S(K)$}

In the setting of Theorem \ref{thm:DH for quotients of cubics local}, we are given a fixed  rank 3 $\QQ_p$-algebra $R_p$ for each $p \in S$ and we are able to compute the average size of of $\Cl(K)_S[2]$ as we range over $\ST$-cubic fields $K$ with a fixed signature such that $K \otimes \QQ_p \simeq R_p$ for every $p \in S$.

As seen in Theorem \ref{thm:DH for quotients of cubics local}, the average size does not depend on the collection of $R_p$, only on $r = \sum_{p \in S} (r_p - 1)$. As a result, we get the following:

\begin{proposition}
\label{prop:fixeds}
Let $s$ be an integer with $|S| \le s \le 3|S|$.
\begin{enumerate}[(i)]
\item The average size of $\Cl(K)_S[2]$ as $K$ ranges over totally real $\ST$-cubic fields with \linebreak $\nu_S(K)\nobreak=\nobreak s$ is equal to $1 + 2^{-(s -|S| + 2)}$,
\item the average size of $\Cl(K)_S[2]$ as $K$ ranges over complex $\ST$-cubic fields with \linebreak $\nu_S(K)\nobreak=\nobreak s$ is equal to $1 + 2^{-(s -|S| + 1)}$, and
\item the average size of $\Cl(K)^+_S[2]$ as $K$ ranges over totally real $\ST$-cubic fields with \linebreak $\nu_S(K)\nobreak=\nobreak s$ is equal to $1 + 2^{-(s -|S|)}$.
\end{enumerate}
\end{proposition}
\begin{proof}
We have $\sum_{p \in S} (r_p - 1) = \nu_S(K) - |S|$. The result then follows directly from Theorem~\ref{thm:DH for quotients of cubics local}.
\end{proof}

We then get the following corollary:

\begin{corollary}
\label{cor:avsselnu}
Let $s$ be an integer with $|S| \le s \le 3|S|$.  When ordered by absolute discriminant,
\begin{enumerate}[(i)]
\item the average value of $2^{s}|\Cl(K)_S[2]|$ as $K$ ranges over totally real cubic fields with \linebreak $\nu_S(K)\nobreak=\nobreak s$ is equal to $2^{s} + 2^{|S|-2}$,
\item the average value of $2^{s}|\Cl(K)_S[2]|$ as $K$ ranges over complex cubic fields with \linebreak $\nu_S(K)\nobreak=\nobreak s$ is equal to $2^{s} + 2^{|S|-1}$, and
\item the average value of $2^{s}|\Cl(K)_S^+[2]|$ as $K$ ranges over totally real cubic fields with \linebreak $\nu_S(K)\nobreak=\nobreak s$ is equal to $2^{s} + 2^{|S|}$.
\end{enumerate}
\end{corollary}

\begin{proof}
This follows immediately from Proposition \ref{prop:fixeds}.
\end{proof}

\subsection{Proof of Theorem  \ref{thm:DH quotients for Sel_2^S}}
We are finally able to calculate the average of $2^{\nu_S(K)}|\Cl(K)_S[2]|$. 
\begin{proposition}
\label{prop:classavg}
When ordered by absolute discriminant,
\begin{enumerate}[(i)]
\item the average of $2^{\nu_S(K)}|\Cl(K)_S[2]|$ as $K$ ranges over totally real $\ST$-cubic fields is equal to 
\linebreak
$\displaystyle{2^{|S|-2} + 2^{|S|} \prod_{p \in S} \left (2 - \frac{1}{p^2 + p + 1} \right )}$,
\item the average of $2^{\nu_S(K)}|\Cl(K)_S[2]|$ as $K$ ranges over comlex $\ST$-cubic fields is
\linebreak
$\displaystyle{2^{|S|-1} + 2^{|S|} \prod_{p \in S} \left (2 - \frac{1}{p^2 + p + 1} \right )}$, and
\item the average of $2^{\nu_S(K)}|\Cl(K)^+_S[2]|$ as $K$ ranges over totally real $\ST$-cubic fields is equal to 
\linebreak
$\displaystyle{2^{|S|} + 2^{|S|} \prod_{p \in S} \left (2 - \frac{1}{p^2 + p + 1} \right )}$.
\end{enumerate}
\end{proposition}

\begin{proof}
We only prove part (i), since the proofs of parts (ii) and (iii) are nearly identical.  For each integer $s$ with $|S| \le s \le 3|S|$, define $\rho(s)$ to be the proportion of $\ST$-cubics $K$ such that $\nu_S(K) = s$. By Corollary \ref{cor:avsselnu}, we then have
\begin{multline}
\label{eq:avgseleq1}
\Avg\left (2^{\nu_S(K)}|\Cl(K)_S[2]|\right)
= \sum_{s = |S|}^{3|S|} \left ( 2^{s} + 2^{|S|-2}\right) \cdot \rho(s) 
= \sum_{s = |S|}^{3|S|} 2^{|S|-2} \cdot \rho(s) + \sum_{s = |S|}^{3|S|} 2^{s} \cdot \rho(s) 
\\
= 2^{|S|-2} + \sum_{s = |S|}^{3|S|} 2^{s} \cdot \rho(s)
= 2^{|S|-2} + 2^{|S|}\prod_{p \in S} \left (2 - \frac{1}{p^2 + p + 1} \right ),
\end{multline}
where the final equality follows from Lemma \ref{lem:avglem}.
\end{proof}

We are now able to prove Theorem \ref{thm:DH quotients for Sel_2^S}.
\begin{proof}[Proof of Theorem  \ref{thm:DH quotients for Sel_2^S}]
As noted in the beginning of this section, we have \linebreak $|\Sel_2^S(K)|\nobreak=\nobreak2^{\nu_S(K) + 3} |\Cl(\cO_K)_S[2]|$ if $K$ is totally real  and $|\Sel_2^S(K)| = 2^{\nu_S(K) + 2}\nobreak|\Cl(\cO_K)_S[2]|$ if $K$ is complex. The result then follows from Proposition \ref{prop:classavg}.
\end{proof}

\section{K-groups}
\label{sec:Kgroups}

We are able to use Theorem \ref{thm:DH quotients for Sel_2^S} to study the average size of $K_{2n}(\cO_K)[2]$ by appealing to the following result of Rognes and Weibel~\cite{RG}.

\begin{theorem}[Theorem 0.7 in \cite{RG}]
\label{thm:RG thm}
Let $K$ be a number field and set $S = \{2\}$. For even $n > 0$, the $2$-rank of $K_{2n}(\cO_K)[2]$ is given by
\begin{equation}
\dimF K_{2n}(\cO_K)[2] = \left \{ \begin{matrix}[ll] 
\dimF \Cl(K)_S[2]  + r_p - 1 &  n \equiv 0 \pmod{4} \\
\dimF \Cl(K)_S[2]  + r_1 + r_p - 1 & n \equiv 1 \pmod{4}\\
\dimF \Cl(K)^+_S[2]  + r_p - 1 &  n \equiv 2 \pmod{4}, n \equiv 3 \pmod{4}
\end{matrix} \right .
\end{equation}
where $r_p$ is the number of places above $2$ in $K$.
\end{theorem}

To prove Theorem \ref{thm:K2thm}, we therefore want to calculate the average value of $2^{r_p}\cdot |\Cl(K)_S[2]|$. 

\pagebreak
\begin{lemma}
\label{lem:avglemK}
Let $S = \{2\}$. When ordered by absolute discriminant,
\begin{enumerate}[(i)]
\item the average value of $2^{r_p}|\Cl(K)_S[2]|$ as $K$ ranges over totally real $\ST$-cubic fields is $\frac{59}{14}$,
\item the average value of $2^{r_p}|\Cl(K)_S[2]|$ as $K$ ranges over complex $\ST$-cubic fields is $\frac{33}{7}$, and 
\item the average value of $2^{r_p}|\Cl(K)^+_S[2]|$ as $K$ ranges over totally real $\ST$-cubic fields $\frac{40}{7}$.
\end{enumerate}
\end{lemma}
\begin{proof}
This follows immediately from Proposition \ref{prop:classavg}.
\end{proof}
We are now able to prove Theorem \ref{thm:K2thm}.

\begin{proof}[Proof of Theorem \ref{thm:K2thm}]
Combine Theorem \ref{thm:RG thm} with Lemma \ref{lem:avglemK}.
 \end{proof}

\end{document}